\documentclass[11pt,a4paper]{amsart}
\usepackage{amsmath,amsthm,amssymb,latexsym}
\usepackage{graphicx}

\newcommand{\bn}{{\mathbb{N}}}

\newcommand{\bp}{{\mathbb{P}}}

\newcommand{\bc}{{\mathbb{C}}}

\newcommand{\ca}{{\mathcal{A}}}

\newcommand{\ch}{{\mathcal{H}}}

\newcommand{\ce}{{\mathcal{E}}}
\newcommand{\cv}{{\mathcal{V}}}

\newcommand{\cc}{{\mathcal{C}}}

\newcommand{\co}{{\mathcal{O}}}

\renewcommand{\l}{\lambda}
\newcommand{\s}{\sigma}

\newcommand{\p}{\varphi}

\renewcommand{\o}{\omega}

\newcommand{\g}{\gamma}
\renewcommand{\gg}{\Gamma}

\newcommand{\ag}{A(\gg)\,}

\DeclareMathOperator{\ran}{rank}

\allowdisplaybreaks
\numberwithin{equation}{section}

\newtheorem{theorem}{Theorem}[section]

\newtheorem{proposition}[theorem]{Proposition}

\theoremstyle{definition}
\newtheorem{definition}[theorem]{Definition}

\newtheorem{example}[theorem]{Example}

\begin{document}



\title[Infinite graphs]
{Spectra of infinite graphs via Schur complement}
\author[L. Golinskii]{L. Golinskii}

\address{Mathematics Division, Institute for Low Temperature Physics and
Engineering, 47 Lenin ave., Kharkov 61103, Ukraine}
\email{golinskii@ilt.kharkov.ua}

\date{\today}

\keywords{Infinite graphs; adjacency operator; spectrum; block matrices; Green function}
\subjclass[2010]{Primary: 05C63; Secondary: 05C76, 47B36, 47B15, 47A10}

\maketitle

\begin{abstract}
The goal of the paper is to apply the general operator theoretic construction known as
the Schur complement for computation of the spectrum of certain infinite graphs which
can be viewed as finite graphs with the ray attached to them. The examples of a multiple
star and a flower with infinite rays are considered.
\end{abstract}

\section*{Introduction and preliminaries}
\label{s0}

\subsection{Graph theory}

We begin with some basics of the graph theory. For the sake of simplicity we restrict ourselves
with simple, connected, undirected, finite or infinite (countable) graphs, although
the main result holds for weighted multigraphs and graphs with loops.
We will label the vertex set $\cv(\gg)$ by positive integers $\bn=\{1,2,\ldots\}$,
$\{v\}_{v\in \cv}=\{j\}_{j=1}^\o$, $\o\le\infty$. The symbol $i\sim j$ means that the vertices
$i$ and $j$ are incident, i.e., $\{i,j\}$ belongs to the edge set $\ce(\gg)$.

The degree (valency) of a vertex $v\in\cv(\gg)$ is a number $\g(v)$ of edges emanating from $v$.
A graph $\gg$ is said to be locally finite, if $\g(v)<\infty$ for all $v\in\cv(\gg)$, and
uniformly locally finite, if $\sup_{\cv}\g(v)<\infty$.

The spectral graph theory studies the spectra and spectral properties of certain
matrices related to graphs (more precisely, operators generated by such matrices in the standard
basis $\{e_k\}_{k\in\bn}$ and acting on the corresponding Hilbert spaces $\bc^n$ or
$\ell^2=\ell^2(\bn)$). One of the most notable of them is the {\it adjacency matrix} $A(\gg)$
\begin{equation}\label{adjmat}
A(\gg)=\|a_{ij}\|_{ij=1}^\o, \quad
a_{ij}=\left\{
  \begin{array}{ll}
    1, & \{i,j\}\in\ce(\gg); \\
    0, & \hbox{otherwise.}
  \end{array}
\right.
\end{equation}
The corresponding adjacency operator will be denoted by the same symbol. It acts as
\begin{equation}\label{adjop}
A(\gg)\,e_k=\sum_{j\sim k} e_j, \qquad k\in\bn.
\end{equation}
Clearly, $\ag$ is a symmetric, densely-defined linear operator, whose domain is the set of all finite
linear combinations of the basis vectors. The operator $A(\gg)$ is bounded and self-adjoint in $\ell^2$,
as long as the graph $\gg$ is uniformly locally finite.

Whereas the spectral theory of finite graphs is very well established (see, e.g.,
\cite{BrHae, Chung97, CDS80}), the corresponding theory for infinite graphs is in its infancy.
We refer to \cite{M82, MoWo89, SiSz} for the basics of this theory. In contrast to the general
consideration in \cite{MoWo89}, our goal is to compute the spectra of certain
infinite graphs (precisely, the spectra of corresponding adjacency operators) which loosely speaking
can be called ``finite graphs with tails attached to them''. To make the notion precise,
we define first an operation of coupling well known for finite graphs
(see, e.g., \cite[Theorem 2.12]{CDS80}).

\begin{definition}\label{coupl}
Let $\gg_k$, $k=1,2$, be two graphs with no common vertices, with the vertex sets and edge
sets $\cv(\gg_k)$ and $\ce(\gg_k)$, respectively, and let $v_k\in \cv(\gg_k)$. A graph
$\gg=\gg_1+\gg_2$ will be called a {\it coupling by means of the bridge $\{v_1,v_2\}$} if
\begin{equation}\label{defcoup}
\cv(\gg)=\cv(\gg_1)\cup \cv(\gg_2), \qquad \ce(\gg)=\ce(\gg_1)\cup \ce(\gg_2)\cup \{v_1,v_2\}.
\end{equation}
So we join $\gg_2$ to $\gg_1$ by the new edge between $v_2$ and $v_1$.
\end{definition}

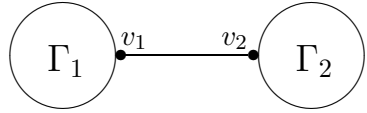
\begin{picture}(300, 100)

\put(120, 50){\circle{80}}   \put(114, 44){\Large{$\Gamma_1$}}

\multiput(142, 50) (50,0) {2} {\circle* {4}}
\put(142, 50) {\line(1,0) {48}}
\put(214, 50){\circle{80}}  \put(208, 44){\Large{$\Gamma_2$}}

\put(142, 54) {$v_1$} \put(180, 54) {$v_2$}

\end{picture}

In general, the adjacency matrix $\ag$ takes the form of a block matrix
\begin{equation}\label{adjcoup}
A(\gg)=\begin{bmatrix}
A(\gg_1) & E \\
E & A(\gg_2)
\end{bmatrix}, \qquad
E=\begin{bmatrix}
1 & 0 & 0 & \ldots \\
0 & 0 & 0 & \ldots \\
0 & 0 & 0 & \ldots \\
\vdots & \vdots & \vdots &
\end{bmatrix}.
\end{equation}
If the graph $\gg_1$ is finite, $V(\gg_1)=\{1,2,\ldots,n\}$, and $V(\gg_2)=\{j\}_{j=1}^\o$,
we can with no loss of generality put $v_1=1$, $v_2=n+1$, so the adjacency matrix $A(\gg)$ can
be written as the block matrix

\begin{equation}\label{adjcoupfin}
A(\gg)=\begin{bmatrix}
A(\gg_1) & E_n \\
E_n & A(\gg_2)
\end{bmatrix}, \qquad
E_n=\begin{bmatrix}
1 & 0 & 0 & \ldots \\
0 & 0 & 0 & \ldots \\
\vdots & \vdots & \vdots & \\
0 & 0 & 0 & \ldots &
\end{bmatrix}.
\end{equation}

If $\gg_2=\bp_\infty$, the one-sided infinite path, we can view the coupling
$\gg=\gg_1+\bp_\infty$ as a finite graph with the tail. Now

\begin{equation}\label{freejac}
A(\gg_2)=
J_0:=
\begin{bmatrix}
 0 & 1 & 0 & 0 & \\
 1 & 0 & 1 & 0 & \\
 0 & 1 & 0 & 1 & \\
     & \ddots  & \ddots & \ddots & \ddots
\end{bmatrix}
\end{equation}
is a Jacobi matrix called a {\it discrete Laplacian} or a {\it free Jacobi matrix}.
It is of particular interest in the sequel.

The spectral theory of infinite graphs with one or several rays attached to certain finite graphs
was initiated in \cite{LeNi-dan, LeNi-umzh, Niz14} wherein a number of particular examples of
graphs was examined. Our argument is based on a general construction from block operator matrices theory
known as the Schur complement. As a matter of fact, the procedure applies not only to adjacency
matrices, but to both Laplacians on graphs of such type.

The examples in the next section rely heavily on the formula of Schwenk \cite{Schw74}
(see \cite[Problem 2.7.9]{CDS80}) for characteristic polynomials of finite graphs $F$
$$ P(\l,F):=\det(\l I-A(F)). $$

Given a graph $F$ and $V\subset\cv(F)$ we denote by $F\backslash V$ the subgraph of $F$
induced by the vertices of $\cv(F)\backslash V$.

{\bf Theorem} (Schwenk). For a given finite graph $F$ and $v\in\cv(F)$, let $\cc(v)$ denote
the set of all simple cycles $Z$ which contain $v$. Then
\begin{equation*}
P(\l,F)=\l P(\l,F\backslash v)-\sum_{v'\sim v}P(\l,F\backslash\{v',v\})-
2\sum_{Z\in\cc(v)}P(\l,F\backslash Z).
\end{equation*}

\subsection{Schur complement}

Let
\begin{equation}\label{blockoper}
\ca=\begin{bmatrix}
A_{11} & A_{12} \\
A_{21} & A_{22}
\end{bmatrix}
\end{equation}
be a block operator matrix which acts in orthogonal sum $\ch_1\oplus\ch_2$ of two Hilbert spaces.
If $A_{11}$ is invertible, the matrix $\ca$ can be factorized as
\begin{equation}\label{factor1}
\ca=\begin{bmatrix}
I & 0 \\
A_{21}A_{11}^{-1} & I
\end{bmatrix}
\begin{bmatrix}
A_{11} & 0 \\
0 & C_{22}
\end{bmatrix}
\begin{bmatrix}
I & A_{11}^{-1}A_{12} \\
0 & I
\end{bmatrix},
\end{equation}
$I$ is the unity operator in the corresponding Hilbert space.
Similarly, if $A_{22}$ is invertible, one can write
\begin{equation}\label{factor2}
\ca=\begin{bmatrix}
I & A_{12}A_{22}^{-1} \\
0 & I
\end{bmatrix}
\begin{bmatrix}
C_{11} & 0 \\
0 & A_{22}
\end{bmatrix}
\begin{bmatrix}
I & 0 \\
A_{22}^{-1}A_{21} & I
\end{bmatrix}.
\end{equation}
Here
\begin{equation}\label{schurcom}
C_{22}:=A_{22}-A_{21}A_{11}^{-1}A_{12}, \qquad C_{11}:=A_{11}-A_{12}A_{22}^{-1}A_{21}
\end{equation}
are usually referred to as the {\it Schur complements} \cite{Schurcomp}.
Both equalities can be checked by multiplication.

The following result is a direct consequence of formulae \eqref{factor1} and \eqref{factor2}.

\begin{proposition}\label{pro1}
Given a block operator matrix $\ca$ $\eqref{blockoper}$, let $A_{22}$ $(A_{11})$ be invertible.
Then $\ca$ is invertible if and only if so is $C_{11}$ $(C_{22})$.
\end{proposition}

Note that in the premises of Proposition \ref{pro1} the inverse $\ca^{-1}$ takes the form
\begin{equation*}
\ca^{-1}=\begin{bmatrix}
C_{11}^{-1} & -C_{11}^{-1}A_{12}A_{22}^{-1} \\
-A_{22}^{-1}A_{21}C_{11}^{-1} & A_{22}^{-1}+A_{22}^{-1}A_{21}C_{11}^{-1}A_{12}A_{22}^{-1}
\end{bmatrix}
\end{equation*}
and, respectively,
\begin{equation*}
\ca^{-1}=\begin{bmatrix}
A_{11}^{-1}+A_{11}^{-1}A_{12}C_{22}^{-1}A_{21}A_{11}^{-1} & -A_{11}^{-1}A_{12}C_{22}^{-1} \\
-C_{22}^{-1}A_{21}A_{11}^{-1} & C_{22}^{-1}
\end{bmatrix}.
\end{equation*}

It looks reasonable applying the latter result to
\begin{equation*}
\l I-\ca=\begin{bmatrix}
\l I-A(\gg_1) & -E \\
-E & \l I-A(\gg_2)
\end{bmatrix}, \qquad \l\in\bc,
\end{equation*}
to obtain

\begin{proposition}\label{pro2}
Given a block operator matrix $\ca$ $\eqref{blockoper}$, let $\l\in\rho(A_{22})$
$\bigl(\l\in\rho(A_{11})\bigr)$.
Then $\l\in\rho(\ca)$ if and only if $\l\in\rho(C_{11})$ $\bigl(\l\in\rho(C_{22})\bigr)$.
\end{proposition}

Here, as usual, $\rho(T)$ stands for the resolvent set of a linear operator $T$, i.e., the set of
complex numbers $\l$ so that $\l I-T$ is boundedly invertible.

\section{Spectra of infinite graphs with tails}

We apply Proposition \ref{pro2} to adjacency matrices \eqref{adjcoup} and \eqref{adjcoupfin} of couplings
$\gg=\gg_1+\gg_2$. It is clear now that
\begin{equation*}
\l I-\ag=(\l I-A(\gg_1))\oplus (\l I-A(\gg_2))+ \Delta, \qquad \ran(\Delta)=2,
\end{equation*}
so we can gather some information about the spectrum of $\ag$ for arbitrary graphs $\gg_j$, $j=1,2$,
from the general perturbation theory of finite rank (see, e.g., \cite[Section 9.2]{BiSo87}). For instance,
according to the famous result of Kato
\begin{equation*}
\s_{ess}(\ag)=\s_{ess}(A(\gg_1))\cup\s_{ess}(A(\gg_2)),
\end{equation*}
(recall that a number $\l\in\s(T)$ belongs to the essential spectrum $\s_{ess}(T)$ of a
self-adjoint operator $T$ if it is not an eigenvalue of finite multiplicity). Moreover,
$\l$ is the eigenvalue of $\ag$ as soon as $\l$ is the isolated eigenvalue of either $A(\gg_1)$ or
$A(\gg_2)$ of multiplicity at least $3$. The Schur complements \eqref{schurcom} take the form
\begin{equation}\label{schcom}
C_{ii}(\l)=\l I-A(\gg_i)-G(\l,\gg_j)\,E, \qquad i\not=j, \quad i,j=1,2,
\end{equation}
where
\begin{equation}\label{green}
G(\l,\gg_j):=\Bigl((\l I-A(\gg_j))^{-1}\Bigr)_{1,1}\,, \quad \l\in\rho(A(\gg_j)), \quad j=1,2
\end{equation}
is the so-called Green's function of $A(\gg_j)$.

We say that $\l$ is a regular point of the graph $\gg$ if $\l\in\rho(\ag)$.

\begin{theorem}\label{th1}
Let $\gg=\gg_1+\gg_2$ be the coupling of a finite graph $\gg_1$ and an arbitrary graph $\gg_2$, and
let $\l$ be a regular point of $\gg_2$. The point $\l$ belongs to the
spectrum of $\gg$  if and only if it solves the equation
\begin{equation}\label{basiceq}
P(\l,\gg_1)-G(\l,\gg_2)P(\l,\gg_1\backslash v_1)=0.
\end{equation}
\end{theorem}
\begin{proof}
By Proposition \ref{pro2} and \eqref{schcom}, $\l\in\s(A(\gg))$ if and only if
\begin{equation*}
\det C_{11}(\l)=\det\bigl((\l I-A(\gg_1)-G(\l,\gg_2)\,E\bigr)=0, \quad
E=\begin{bmatrix}
1 & 0 &  \ldots & 0 \\
0 & 0 &  \ldots & 0 \\
\vdots & \vdots &  & \vdots \\
0 & 0 &  \ldots & 0
\end{bmatrix}.
\end{equation*}
The result follows by expanding the latter determinant along the first row.
\end{proof}

The basic example for us is $\gg_2=\bp_\infty$. In this case
\begin{equation*}
\s(\gg)=\s_{ess}(\gg)\cup\s_d(\gg)=[-2,2]\cup\s_d(\gg),
\end{equation*}
where the discrete spectrum $\s_d(\gg)$ is the set of eigenvalues of finite
multiplicity off $[-2,2]$.  Theorem \ref{th1} can be applied, and the Green's function is
known explicitly (see, e.g., \cite{KiSi03})
\begin{equation}\label{resfree}
\Bigl(\l I-A(\bp_\infty)\Bigr)^{-1}=\|r_{ij}(z)\|_{i,j=1}^\infty, \ \
r_{ij}(z)=\frac{z^{i+j}-z^{|i-j|}}{z-z^{-1}}\,,  \ \ \l=z+\frac1{z}\,,
\end{equation}
$|z|<1$, and so
\begin{equation}\label{greenfree}
G(\l,\bp_\infty)=r_{11}(z)=z.
\end{equation}
The discrete spectrum agrees with the zero set of algebraic equation \eqref{basiceq}
\begin{equation}\label{algeq}
\l\in\s_d(\gg) \Leftrightarrow P(\l,\gg_1)-x P(\l,\gg_1\backslash v_1)=0, \ \ \l=x+\frac1{x}\,,
\ \ x\in(-1,1).
\end{equation}
So the problem amounts to computation of two characteristic
polynomials and solving \eqref{algeq}.

\begin{example} ``A multiple star''.

Let $\kappa:=(k_1,k_2,\ldots,k_n)$, $k_j\in\bn$ be an $n$-tuple of positive integers.
Denote by $S(\kappa)$ the graph obtained from the standard star graph $K_{1,n}$ with
$n\ge2$ edges by inserting $k_j-1$ new vertices into $j$'s edge, so this edge contains $k_j+1$
vertices altogether. Put $\gg_1=S(\kappa)$ and consider the coupling $\gg=\gg_1+\bp_\infty$
obtained from $S(\kappa)$ by attaching the infinite path to its root $v_1$.

Denote by $Q(\cdot,m)$ the characteristic polynomial of the finite path $\bp_m$ with $m$ vertices, so
$$ Q(\l,m)=U_m\Bigl(\frac{\l}2\Bigr), \qquad U_m(\cos t):=\frac{\sin(m+1)t}{\sin t} $$
is the standard Chebyshev polynomial of the second kind.

To compute the characteristic polynomial of $\gg_1$ we apply the result of Schwenk which now looks
$$ P(\l,\gg_1)=\l\,Q(\l)-Q(\l)\,\sum_{j=1}^n \frac{Q(\l,k_j-1)}{Q(\l,k_j)},
\quad Q(\l):=\prod_{j=1}^n Q(\l,k_j). $$
It is clear that
$$ P(\l,\gg_1\backslash v_1)=Q(\l), $$
so the equation in \eqref{algeq} takes the form
$$ Q(\l)\,\Bigl\{\frac1{x}-\sum_{j=1}^n \frac{Q(\l,k_j-1)}{Q(\l,k_j)}\Bigr\}=0, \quad \l=x+\frac1{x}\,. $$
Since $Q\not=0$ off $[-2,2]$ we come to
$$ \sum_{j=1}^n \frac{Q(\l,k_j-1)}{Q(\l,k_j)}-\frac1{x}=0, \quad \l=x+\frac1{x}\,, \quad -1<x<1. $$
The function in the left-hand side is odd (as a function of $x$), so we can restrict ourselves
with the values $0<x<1$. Putting $x=e^{-t}$, $t>0$, we obtain after a bit of calculation
\begin{equation}\label{chareq}
\p(t):=\sum_{j=1}^n\frac{\sinh k_j t}{\sinh(k_j+1)t}=e^t, \qquad t>0.
\end{equation}

Since the function
$$ \frac{\sinh at}{\sinh bt}, \qquad 0<a<b, \quad t>0 $$
is easily seen to be monotone decreasing (it follows, e.g., from the infinite product expansion
of $\sinh z$), then so is $\p$ in the left-hand side of \eqref{chareq}. Moreover, $\p$ vanishes
at infinity. Next,
$$ \lim_{t\to 0} \p(t)=\sum_{j=1}^n\frac{k_j}{k_j+1}>1 $$
(we discard the trivial configuration $n=2$, $k_1=k_2=1$). So \eqref{chareq} has a unique
solution $t_+>0$. Finally, the discrete spectrum is
\begin{equation}\label{disspec1}
\s_d(\gg)=\pm\l_+, \qquad \l_+:=2\cosh t_+.
\end{equation}

In particular case $k_j=p$, $1\le j\le n$, equation \eqref{chareq} looks
$$ n\,\frac{\sinh pt}{\sinh (p+1)t}=e^t, \quad n\bigl(e^{pt}-e^{-pt}\bigr)=e^{(p+2)t}-e^{-pt}\,, $$
which is equivalent to
$$ (n-1)x^{2p+2}-nx^2+1=0. $$

Note that in this case the more detailed description of the spectrum is available,
see \cite[Example 2.3]{Gol1}. Precisely, there are $p$ eigenvalues lying on the essential
spectrum $[-2,2]$.
\end{example}

\begin{example} ``A flower with $n$ petals''.

In this example $\gg_1$ is composed of $n$ cycles $\{\bc_j\}_{j=1}^n$, glued together at one
common vertex (root) $\co$. Put $\gg=\gg_1+\bp_\infty$ with the infinite path attached to the root $\co$.
Assume that the cycle $\bc_j$ contains $k_j+1$ vertices.

To compute the characteristic polynomial of $\gg_1$ we apply again the Schwenk theorem which gives
$$ P(\l,\gg_1)=Q(\l)\,\Bigl\{\l-2\sum_{j=1}^n \frac{Q(\l,k_j-1)+1}{Q(\l,k_j)}\Bigr\}.  $$
As in the above example, $P(\cdot,\gg_1\backslash\co)=Q$, and we come to the following equation
$$ 2\sum_{j=1}^n \frac{Q(\l,k_j-1)+1}{Q(\l,k_j)}-\frac1{x}=0, \quad \l=x+\frac1{x}\,, \quad -1<x<1. $$

Putting $x=e^{-t}$, $t>0$, we obtain
\begin{equation}\label{chareq1}
2\sum_{j=1}^n\frac{\sinh k_j t}{\sinh(k_j+1)t}=e^t, \qquad t>0.
\end{equation}
The same argument as above shows that \eqref{chareq1} has a unique solution $t_+>0$, so the point
$\l_+=2\cosh t_+\in\s_d(\gg)$ for {\it all} configurations in $\gg_1$ (with no exceptions).
Putting $x=-e^{-t}$, $t>0$, we come to
\begin{equation}\label{chareq2}
\p(t):=2\sum_{j=1}^n\frac{\sinh k_j t+(-1)^{k_j+1}\,\sinh t}{\sinh(k_j+1)t}=e^t, \qquad t>0.
\end{equation}
Since
$$ \frac{\sinh mt-\sinh t}{\sinh (m+1)t}=\frac{\sinh\frac{m-1}2 t}{\sinh\frac{m+1}2 t}\,, $$
$\p$ is a monotone decreasing function vanishing at infinity, and $\p(0+)>1$ for all configurations
in $\gg_1$. Hence \eqref{chareq2} has a unique solution $t_->0$, and $\l_-=-2\cosh t_-\in\s_d(\gg)$.
Finally, the discrete spectrum is
\begin{equation}\label{disspec2}
\s_d(\gg)=\l_{\pm}, \qquad \l_{\pm}:=\pm 2\cosh t_{\pm}.
\end{equation}

Note that in particular case $n=2$, $k_1=k_2$ (the propeller with equal blades), a complete
description of the spectrum is given in \cite[Example 3.4]{Gol1}.

\end{example}


\begin{thebibliography}{99}






\bibitem{BiSo87}
M. Birman and M. Solomjak, Spectral Theory of Self-Adjoint Operators in Hilbert Spaces,
D.Reidel Publishing Company, Dordrecht, 1987.

\bibitem{BrHae}
A. E. Brouwer, W. H. Haemers, Spectra of Graphs, Springer, Universitext, 2012.

\bibitem{Chung97}
F. Chung, Spectral graph theory, volume 92 of CBMS Regional Conference
Series in Mathematics. Published for the Conference Board of the Mathematical
Sciences, Washington, DC, 1997

\bibitem{CDS80}
D. M. Cvetkovi\'c, M. Doob, H. Sachs, Spectra of Graphs -- Theory and Applications, Academic Press, 1980.




\bibitem{Gol1}
L. Golinskii, Spectra of infinite graphs with tails, preprint arxiv:1503.04952.

\bibitem{KiSi03}
R. Killip, B. Simon, Sum rules for Jacobi matrices and their applications to spectral theory,
Ann. Math., {\bf 158} (2003), 253--321.



\bibitem{LeNi-dan}
V. Lebid, L. Nizhnik, Spectral analysis of locally finite graphs with one infinite chain,
Proc. Ukranian Academy of Sci., (2014), no.3, 29--35.

\bibitem{LeNi-umzh}
V. Lebid, L. Nizhnik, Spectral analysis of certain graphs with infinite chains, Ukr. J. Math. {\bf 66}
(2014), no.9, 1193--1204.


\bibitem{M82}
B. Mohar, The spectrum of an infinite graph, Linear Alg. Appl., {\bf 48} (1982), 245--256.

\bibitem{MoWo89}
B. Mohar, W. Woess, A survey on spectra of infinite graphs, Bull. London Math. Soc., {bf 21} (1989), 209--234.

\bibitem{Niz14}
L. P. Nizhnik, Spectral analysis of metric graphs with infinite rays, Methods of Func. Anal. and Topology,
{\bf 20} (2014), 391--396.

\bibitem{Schurcomp}
I. Schur, \"{U}ber Potenzreihen, die im Innern des Einheitskreises beschr\"{a}nkt sind, I, J. Reine Angew. Math. 
{\bf 147} (1917), 205--232.

\bibitem{Schw74}
A. Schwenk, Computing the characteristic polynomial of a graph. -- In: Graphs and combinatorics
Lect. Notes Math., {bf 406} (1974), 153--172.


\bibitem{SiSz}
B. Simon, {\it Szeg\H{o}'s Theorem and its Descendants}, Princeton Uiversity Press, 2011.



\end{thebibliography}
\end{document}